\documentclass[english,11pt,reqno,twoside]{amsart}
\usepackage[T1]{fontenc}
\usepackage[utf8]{inputenc}
\usepackage{amsmath}
\usepackage{amsthm}
\usepackage{amssymb}

\usepackage{soul,color}
\usepackage{t1enc}
\usepackage{a4,indentfirst,latexsym,graphics}

\usepackage{mathrsfs}
\usepackage[noadjust]{cite}
\usepackage{enumitem,graphicx}
\usepackage[colorlinks=true,urlcolor=blue,
citecolor=red,linkcolor=blue,linktocpage,pdfpagelabels,
bookmarksnumbered,bookmarksopen]{hyperref}
\usepackage[english]{babel}

\usepackage[left=2.1cm,right=2.1cm,top=2.72cm,bottom=2.72cm]{geometry}

\usepackage[metapost]{mfpic}

\usepackage[hyperpageref]{backref}
\usepackage[colorinlistoftodos]{todonotes}
\usepackage[normalem]{ulem}
\usepackage{braket}
\usepackage{fancyhdr}

\makeatletter
\providecommand\@dotsep{5}
\def\listtodoname{List of Todos}
\def\listoftodos{\@starttoc{tdo}\listtodoname}
\makeatother

\makeatletter
\def\namedlabel#1#2{\begingroup
    #2%
    \def\@currentlabel{#2}%
    \phantomsection\label{#1}\endgroup
}
\makeatother

\newtheorem{Th}{Theorem}[section]
\newtheorem{Prop}[Th]{Proposition}
\newtheorem{Lem}[Th]{Lemma}

\theoremstyle{definition}
\newtheorem{Def}[Th]{Definition}
\newtheorem{Ex}[Th]{Example}
\theoremstyle{remark}
\newtheorem{Rem}[Th]{Remark}

\newcommand{\R}{\mathbb{R}}

\newcommand{\cA}{\mathcal{A}}
\newcommand{\cB}{\mathcal{B}}

\newcommand{\cC}{\mathcal{C}}

\newcommand{\dV}{\mathrm{d}v_g}

\newcommand{\Vol}{\mathrm{vol}_g}

\DeclareMathOperator*{\essinf}{ess\,inf}
\DeclareMathOperator*{\supp}{supp}

\definecolor{yellow-green}{rgb}{0.6, 0.8, 0.2}

\numberwithin{equation}{section}

\title[Finite-energy solutions to Einstein-scalar field Lichnerowicz equations]{Finite-energy solutions to Einstein-scalar field Lichnerowicz equations on complete Riemannian manifolds}

\author[Bartosz Bieganowski]{Bartosz Bieganowski}
\author[Pietro d'Avenia]{Pietro d'Avenia}
\author[Jacopo Schino]{Jacopo Schino}

\address[B. Bieganowski]{\newline\indent
	Faculty of Mathematics, Informatics and Mechanics, \newline\indent
	University of Warsaw, \newline\indent
	ul. Banacha 2, 02-097 Warsaw, Poland}
\email{\href{mailto:bartoszb@mimuw.edu.pl}{bartoszb@mimuw.edu.pl}}

\address[P. d'Avenia]{\newline\indent
	Dipartimento di Meccanica, Matematica e Management, \newline\indent
	Politecnico di Bari, \newline\indent
	Via E. Orabona 4, 70125 Bari, Italy}
\email{\href{mailto:pietro.davenia@poliba.it}{pietro.davenia@poliba.it}}

\address[J. Schino -- corresponding author]{\newline\indent
	Faculty of Mathematics, Informatics and Mechanics, \newline\indent
	University of Warsaw, \newline\indent
	ul. Banacha 2, 02-097 Warsaw, Poland \newline\indent
	and \newline\indent
	Neuromed Mediterranean University of Sciences and Technologies \newline\indent
	Via dell'Elettronica 1, 86077 Pozzilli (IS), Italy}
\email{\href{mailto:jacopo.schino@unineuromed.it}{jacopo.schino@unineuromed.it}}



\begin{document}
\begin{abstract}
We consider the singular elliptic problem of the form
\[
-\Delta u + V(x)u = \mathcal B(x)|u|^{2^*-2}u + \frac{\mathcal A(x)}{|u|^{2^*}u},
\qquad u\in H^1(M),
\]
where the coefficients are allowed to have low regularity. Under natural spectral assumptions on \(-\Delta+V\), geometric assumptions on the manifold $M$ ensuring the Sobolev embedding \(H^1(M)\hookrightarrow L^{2^*}(M)\), and a suitable global integrability/smallness condition involving \(\mathcal A\), \(\mathcal B\), and a function \(\psi \in H^1(M)\), we prove the existence of a nonnegative finite-energy supersolution. If, in addition, the Ricci curvature is nonnegative and \(\mathcal B\ge 0\), we obtain a positive finite-energy solution. The proof relies on a family of \(\varepsilon\)-regularized problems, mountain pass arguments, and a limiting procedure in which Harnack’s inequality plays a crucial role in handling the singular term on noncompact manifolds. We also prove a nonexistence result showing that the global integrability condition on \(\mathcal A\) is, in a precise sense, necessary for the existence of nonnegative supersolutions.
\end{abstract}
\maketitle

\section{Introduction}

The Einstein field equations, introduced by A. Einstein (in their original form) in 1915 (see \cite{Einstein1, Einstein2}), describe the relation between the geometry of spacetime and the distribution of matter in this spacetime. Einstein-scalar field theories have seen significant developments in recent years, including attempts to explain the acceleration of the expansion of the universe (see \cite{Rendall1, Rendall2}).

The initial data in the analysis of the Einstein-scalar field equations, whose solutions describe the dynamics of the spacetime, must satisfy certain constraint conditions, which are the Gauss and Codazzi equations.
More precisely, let $(M, \widetilde{g})$ be a Riemannian manifold of dimension $N \geq 3$ without boundary. The Gauss and Codazzi equations read as
\begin{align*}
\mathrm{Scal}_{\widetilde{g}} - |K|_{\widetilde{g}}^2 + (\mathrm{Tr}_{\widetilde{g}} K)^2 - 2 \rho = 0,\\
\nabla_{\widetilde{g}} \cdot K - \nabla_{\widetilde{g}} \mathrm{Tr}_{\widetilde{g}} K - J = 0,
\end{align*}
where $K$ is the second fundamental form, $\mathrm{Scal}_{\widetilde{g}}$ denotes the scalar curvature, $\rho \in \R$, and $J$ is a vector field on $M$. Using the conformal method (see \cite{Lichnerowicz}), i.e., changing the metric $\widetilde{g} = u^{\frac{4}{N-2}} g$, the Gauss and Codazzi equations may be rewritten as the Hamiltonian and momentum constraints (see e.g. \cite{Benalili} for detailed computations). The Hamiltonian constraint is then of the form
\begin{equation}\label{eq:1}
-\Delta_g u + V u = \cB u^{2^*-1} + \frac{\cA}{u^{2^*+1}}, \quad u > 0,
\end{equation}
which is called the \textit{Einstein-scalar field Lichnerowicz equation}, where $\Delta_g$ is the Laplace-Beltrami operator on the manifold $(M,g)$ and $2^* := \frac{2N}{N-2}$.

A good understanding of the behaviour of equation \eqref{eq:1} is crucial, and the existence and qualitative properties of positive solutions to \eqref{eq:1} play a fundamental role.
In the case of a compact Riemannian manifold $(M,g)$, equation  \eqref{eq:1} and its variants have been widely studied and are well understood \cite{Benalili, ChoquetBruhat, Hebey, Isenberg1995, MaWei, Ngo, Premoselli, BK, BDSS,Chrusciel_Gicquaud,PremoselliDCDS}. However, the analysis of \eqref{eq:1} in the case of a noncompact domain $M$ is a challenging problem. One can look for solutions using the sub- and supersolutions method (see e.g. \cite{MaWei}), but in this case, one does not find finite-energy solutions (see Definition \ref{def:solution} below). The noncompact case was studied also in \cite{Albanese}, where asymptotically Euclidean manifolds were considered.

We mention here that some variants of singular-type equations \eqref{eq:1}, equipped with appropriate boundary conditions, were also studied in the case $M = \Omega \subset \R^N$, where $\Omega$ is a bounded domain. We refer the reader to \cite{ArcoyaBoccardo, Arcoya, Boccardo, Durastanti, OlivaPetitta}.

Having this motivation in mind, in this paper, we fix a complete, smooth, $N$-dimensional Riemannian manifold $(M,g)$ without boundary, with $N \geq 3$. Let us underline that $M$ may be noncompact.
We consider the following equation
\begin{equation}\label{eq:main}
-\Delta u + V(x)u = \cB(x) |u|^{2^*-2} u + \frac{\cA(x)}{|u|^{2^*} u}, \quad u \in H^1 (M),
\end{equation}
where $H^1(M)$ denotes the space of functions $u \in L^2(M)$ such that $|\nabla u| \in L^2(M)$ -- cf. Subsection \ref{SS:SobSp}. Here and in the sequel, $\nabla u$ denotes the covariant derivative of $u$, $|\nabla u|^2 = g^{ij} (\nabla u)_i (\nabla u)_j$, and $\Delta$ the Laplace--Beltrami operator.

We impose the following assumptions.
\begin{itemize}
    \item[(\namedlabel{AssA}{A})] $\cA \in L^1_\textup{loc}(M)$, $\cA \geq 0$, and $\cA \not= 0$.
    \item[(\namedlabel{AssB}{B})] $\cB \in L^{2^*}_\textup{loc}(M)$ and $\cB_+ := \max\{\cB,0\} \not = 0$.
    \item[(\namedlabel{AssV}{V})] $V \in L^\infty (M)$ and $\inf \sigma(-\Delta + V) > 0$, where $\sigma$ denotes the spectrum of the operator $-\Delta + V$ on $L^2(M)$.
\end{itemize}
Clearly, (\ref{AssV}) is satisfied if $V \in L^\infty (M)$ and $\essinf V > 0$.

We collect some definitions that will be used throughout the paper.

\begin{Def}\label{def:solution}
If $u\in H^1(M)$ satisfies, for every $\varphi\in \cC_0^\infty(M)$,
\begin{equation}\label{eq:req}
\int_{M} \left[ |\cB(x)| |u|^{2^*-1} |\varphi| +\frac{\cA(x) |\varphi|}{|u|^{2^* + 1}} \right] \, \dV < +\infty,
\end{equation}
where 
$$
\int_{M} \frac{\cA(x) \varphi}{|u|^{2^*} u}  \,\dV := \int_{\supp\cA} \frac{\cA(x) \varphi}{|u|^{2^*} u}  \,\dV,
$$
then we say that $u$ is a {\em supersolution} of \eqref{eq:main} if and only if, for every $\varphi\in \cC_0^\infty(M)$ with $\varphi \ge 0$,
\begin{equation}\label{eq:supsol}
\int_{M} \nabla u \nabla \varphi \,\dV
+ \int_{M} V(x)u \varphi \,\dV
\ge
\int_{M}\cB(x) |u|^{2^*-2} u \varphi \,\dV
+ \int_{M} \frac{\cA(x) \varphi}{|u|^{2^*} u}  \,\dV;
\end{equation}
likewise, we say that $u$ is a {\em solution} of \eqref{eq:main} if and only if, for every $\varphi\in \cC_0^\infty(M)$, equality holds in \eqref{eq:supsol}.

Next, we say that $u \in H^1(M)$ has {\em finite energy} if and only if
\begin{equation*}
\int_{M} \left[ |\cB(x)| |u|^{2^*} + \frac{\cA(x)}{|u|^{2^*}} \right] \, \dV < +\infty.
\end{equation*}

Finally, we say that $u\in H^1(M)$ is {\em positive} if and only if $\essinf_B u > 0$ for every geodesic ball $B \subset M$.
\end{Def}

Note that the finiteness of
\[
\int_{M}  |\cB(x)| |u|^{2^*-1} |\varphi|  \, \dV 
\]
in \eqref{eq:req} is justified because, as recalled in Subsection \ref{SS:SobSp}, the embedding $H^1(M) \hookrightarrow L^{2^*}(M)$ need not hold.

At least formally, finite-energy solutions to \eqref{eq:main} are critical points of the energy functional
\begin{equation}\label{eq:I}
I(u) = \int_{M} |\nabla u|^2 + V(x) u^2 \, \dV -\frac{1}{2^*} \int_{M} \cB(x) |u|^{2^*} \,\dV
+\frac{1}{2^*} \int_{M} \frac{\cA(x)}{|u|^{2^*}} \,\dV.
\end{equation}
Since there are functions in $H^1(M)$ such that 
$$
\int_{M} \frac{\cA(x)}{|u|^{2^*}} \,\dV = +\infty,
$$
following \cite{Hebey}, we introduce an approximated problem and solve it via variational methods.

Throughout the manuscript, $S > 0$ denotes the optimal constant (which depends on $V$) in the embedding $H^1(M) \hookrightarrow L^{2^*}(M)$ whenever it holds, cf. Subsection \ref{SS:SobSp}.
Moreover, for a function $v \colon M \rightarrow \R$, $v_+ := \max\{v, 0\}$ (respectively, $v_- = \max\{ -v, 0\} $) denotes the positive (respectively, negative) part of $v$, and $B(x,R)$ denotes the geodesic ball centered at $x \in M$ of radius $R>0$.

Our first result reads as follows.

\begin{Th}\label{thm:main}
Assume that the Ricci curvature of $(M,g)$ is bounded from below and
\begin{equation}\label{VolumesPositiveNew}
\inf_{x \in M} \Vol (B(x,1)) > 0.
\end{equation}
If (\ref{AssA}), (\ref{AssB}), and (\ref{AssV}) hold and there exist $K>0$, $\Theta >0$, and $\psi \in H^1 (M)$ such that
\begin{equation}\label{eq:B+}
\cB \in L^\infty(\supp \psi) \quad \text{and} \quad \int_{M} \cB(x) |\psi|^{2^*} \, \dV > 0,
\end{equation}
\begin{equation}\label{eq:psiK}
\|\psi\|^{2^*} \int_{M} \frac{\cA(x)}{|\psi|^{2^*}} \, \dV \leq \frac{K}{(|\cB_+|_{L^\infty(\supp \psi)} S)^{N-1}},
\end{equation}
and
\begin{equation}\label{eq:ThetaK}
N \Theta^2 + (N-2) \frac{|\cB_-|_{L^\infty(\supp \psi)}}{|\cB_+|_{L^\infty(\supp \psi)}} \Theta^{2^*} + (N-2) \frac{K}{\Theta^{2^*}} \le 2,
\end{equation}
then \eqref{eq:main} admits a nonnegative finite-energy supersolution $u \in H^1 (M)$. 
If, moreover, the Ricci curvature of $(M,g)$ is nonnegative and $\cB \ge 0$, then $u$ is a positive solution to \eqref{eq:main}.
\end{Th}

\begin{Rem}\label{rem:T12}We observe the following facts.
\begin{enumerate}[label=(\roman*),ref=\roman*]
    \item For a given $\psi \in H^1(M)$ that verifies \eqref{eq:B+}, condition \eqref{eq:ThetaK} is always satisfied by possibly non-optimal $\Theta$ and $K$; see also Subsection \ref{ss:alt} for a version of \eqref{eq:ThetaK} that does not depend on $\psi$. On the other hand, when $\cB \ge 0$ in $\supp \psi$, for a given $K > 0$, the minimum of the left-hand side of \eqref{eq:ThetaK} is attained at $\Theta = K^\frac{N-2}{4(N-1)}$ and equals $2 (N-1) K^\frac{N-2}{2(N-1)}$. Thus, \eqref{eq:ThetaK} holds if and only if $K \le \bigl( \frac{1}{N-1} \bigr)^\frac{2(N-1)}{N-2}$. Considering \eqref{eq:psiK}, the optimal value for $K$ is $\bigl( \frac{1}{N-1} \bigr)^\frac{2(N-1)}{N-2}$, which gives a wider range than \cite[formula (3.3)]{Hebey}, where, with the notations of this paper, $K = \bigl( \frac{1}{2(N-1)} \bigr)^\frac{N}{N-2} \frac{1}{N-2}$.
    \item The reason why we need the assumption $\cB \ge 0$ to find a solution to \eqref{eq:main} is related to the fact that the $L^\infty$ norm of a solution to a critical equation need not be controlled by its $H^1$ norm. See Remark \ref{rem:Linfinity} for further details.
    \item \label{i4remT12} The integrability of $\cA |\psi|^{-2^*}$, which is a consequence of \eqref{eq:psiK}, is a sort of global condition on $\cA$, which is not ensured by (\ref{AssA}) or its stronger variant (\ref{AssA1}) below. It is also a necessary condition for the existence of nonnegative supersolutions, as stated in Theorem \ref{th:non} below. 
\end{enumerate}
\end{Rem}

Let us discuss conditions \eqref{eq:B+}-\eqref{eq:psiK} in a series of examples. Specifically, Example \ref{ex:Bloc} demonstrates that such conditions may hold without global assumptions about $\cB$. Furthermore, even in the case of a compact manifold $M$, (\ref{AssA}) and (\ref{AssB}) are weaker than what one would typically need in a variational approach, cf. (\ref{AssA1}) and (\ref{AssB1}); in particular, we require far less than $\cA$, $\cB$, and $V$ to be smooth, as done in \cite{Hebey}.

\begin{Ex}\label{ex:RN}
Suppose that $M = \R^N$ equipped with the Euclidean metric. In addition to (\ref{AssA}) and (\ref{AssB}), assume that $\cB \in L^\infty (\R^N)$, $\cB \geq 0$, and $\cA$ decays exponentially at infinity. Then, every positive $\psi \in H^1(M)$ with a polynomial decay at infinity satisfies \eqref{eq:B+} and $\cA |\psi|^{-2^*} \in L^1(M)$. In this case, \eqref{eq:psiK} holds if $|\cB|_{L^\infty}$ is sufficiently small or 
$\cA = \theta \cA_0$ with $\cA_0|\psi|^{-2^*} \in L^1(M)$ fixed and $\theta > 0$ sufficiently small.
\end{Ex}

\begin{Ex}\label{ex:Bloc}
Assume that the Ricci curvature of $(M,g)$ is bounded from below and \eqref{VolumesPositiveNew} holds. Let $B_1, B_2 \subset M$ be two open geodesic balls such that $\overline{B_1} \subset B_2$. In addition to (\ref{AssA}) and (\ref{AssB}), assume that $\cA$ is supported in $\overline{B_1}$, and take $\psi \in H^1(M)$ supported in $\overline {B_2}$ such that $\essinf_{\supp \cA} \psi > 0$. Then, $\cA |\psi|^{-2^*} \in L^1(M)$. If $\cB$ is nonnegative and essentially bounded in $B_2$, then \eqref{eq:B+} holds, and \eqref{eq:psiK} is satisfied if $\cA$ or $\cB$ are sufficiently small in the sense of Example \ref{ex:RN}.
If $\cB$ is essentially bounded but takes both positive and negative values in $B_2$, then $\psi$ can be suitably chosen for \eqref{eq:B+} to hold.
\end{Ex}

\begin{Ex}
If $M$ has infinite measure and $\liminf_{x \to \infty} \cA(x) > 0$ (e.g., if $\cA$ is constant), then, for every $v \in H^1(M)$, $\cA |v|^{-2^*} \not \in L^1(M)$.
\end{Ex}

Let us briefly illustrate our strategy to prove Theorem \ref{thm:main}. Due to the singular term in \eqref{eq:main}, as anticipated earlier, we perturb the corresponding integral in $I$, defined in \eqref{eq:I}, via a parameter $\varepsilon > 0$, obtaining a family of auxiliary problems. At this stage, we work with the stronger assumptions (\ref{AssA1}) and (\ref{AssB1}), which allow a rigorous variational formulation. After obtaining a family of solutions to the auxiliary problems, we pass to the limit as $\varepsilon \to 0^+$; here, to prove that the weak limit is a solution to the original problem, Harnack's inequality (Proposition \ref{Pr:Harnack}) is crucial. We highlight that the possible noncompactness of $M$ prevents us from using the same argument as in \cite{Hebey}, where a uniform positive pointwise lower bound was obtained for such a family of solutions. Finally, we approximate the functions $\cA$ and $\cB$ satisfying (\ref{AssA}) and (\ref{AssB}) with a sequence of functions satisfying (\ref{AssA1}) and (\ref{AssB1}) and pass again to the limit.

We conclude this section with the nonexistence result mentioned in Remark \ref{rem:T12} (\ref{i4remT12}).

\begin{Th}\label{th:non}
Assume that (\ref{AssA}), (\ref{AssB}), and (\ref{AssV}) hold. If
\begin{equation*}
\int_{M} \frac{\cA(x)}{|v|^{2^*}} \, \dV = +\infty \quad \text{for every } v \in H^1(M),
\end{equation*}
then no nonnegative supersolutions $u \in H^1(M)$ to \eqref{eq:main} such that $\cB \ge 0$ in $\supp u$ exist. If, moreover, the Ricci curvature is bounded below and \eqref{VolumesPositiveNew} holds, then no nonnegative supersolutions $u \in H^1(M)$ to \eqref{eq:main} such that $\cB_- \in L^\infty(\supp u)$ exist.
\end{Th}

The paper is structured as follows. In Section \ref{S:prel}, we provide the functional-analytic setting of \eqref{eq:main} and Liouville-type theorems for nonnegative supersolutions. In Section \ref{S:eps}, we prove Theorem \ref{thm:main} under the stronger assumptions (\ref{AssA1}) and (\ref{AssB1}), making make use of the approximating problems mentioned above. In Section \ref{Sec:pfmain}, we prove Theorems \ref{thm:main} and \ref{th:non}.

\section{Preliminaries}\label{S:prel}
In this section, we first introduce the Sobolev space $H^1(M)$ and provide some general and useful properties about it. Then, we present some qualitative aspects of supersolutions to \eqref{eq:main}.

\subsection{Sobolev spaces on manifolds}\label{SS:SobSp}

We introduce the space $H^1 (M)$ as the completion of 
$\cC_0^\infty(M)$, the space of $\cC^\infty$ functions with compact support in $M$, with respect to the norm
$$
\|u\|_{H^1} := \left( \int_M |\nabla u|^2 + u^2 \, \dV \right)^{1/2}.
$$
It is a Hilbert space \cite[Proposition 2.2]{HebeyBook} with respect to the scalar product
$$
\langle u, v \rangle_{H^1} := \int_M \nabla u \nabla v + uv \, \dV.
$$
In view of (\ref{AssV}), we can consider the equivalent norm
\begin{equation*}
\|u\| := \left(\int_{M} |\nabla u|^2 + V(x) u^2 \, \dV\right)^{1/2}
\end{equation*}
and the associated scalar product
\begin{equation*}
\langle u, v \rangle = \int_{M} \nabla u \nabla v  + V(x) uv \, \dV.
\end{equation*}

It is known that the embedding $H^1 (M) \hookrightarrow L^{2^*} (M)$ may not hold; see, e.g., \cite[Proposition 3.3]{HebeyBook}. However, there is a characterization of complete manifolds for which such an embedding holds (\cite[Theorem 3.18]{HebeyBook}). Assume that the Ricci curvature is bounded from below. Then, the embedding above holds if and only if \eqref{VolumesPositiveNew} is satisfied. In particular, Sobolev embeddings hold true if the Ricci curvature is bounded from below and the injectivity radius is positive (\cite[Corollary 3.19]{HebeyBook}, \cite[Theorem 2.21]{AubinBook}).
Finally, $S > 0$ denotes the optimal constant such that
$$
|\varphi|_{L^{2^*}}^{2^*} \leq S \|\varphi\|^{2^*} \quad \text{for every } \varphi \in H^1 (M), 
$$
i.e.,
$$
S := \inf_{\varphi \in H^1(M) \setminus \{0\}} \frac{|\varphi|_{L^{2^*}}^{2^*}}{\|\varphi\|^{2^*}}.
$$

\subsection{Liouville-type theorems}

\begin{Prop}[Harnack's inequality]\label{Pr:Harnack}
Let $\omega \ge 0$ and assume that the Ricci curvature of $(M,g)$ is nonnegative. If $B(R) \subset M$ is a geodesic ball of geodesic radius $R>0$ and $q>0$ is sufficiently small, then there exists $c>0$ such that, for every $u \in H^1(M)$ satisfying
$$
u \ge 0
\text{ and }
-\Delta u + \omega u \ge 0 \quad \text{in } B(R),
$$
there holds
\begin{equation*}
\left( \int_{B(R/8)} u^q \, \dV \right)^{1/q} \le c \inf_{B(R/16)} u.
\end{equation*}
\end{Prop}
\begin{proof}
The inequality follows as in \cite[page 342]{Zhang_2000_KMJ}. Indeed, since the Ricci curvature is nonnegative, from the Gromov--Bishop volume comparison theorem and the results in \cite{Saloff-Coste_1992_JDE}, the assumptions of \cite[Lemma 3.1]{Zhang_2000_KMJ} are satisfied.
\end{proof}

\begin{Rem}
The same result holds if $M$ has a boundary, provided that $B(R) \cap \partial M = \emptyset$ (see \cite[Lemma 3.1]{Zhang_2000_KMJ}).
\end{Rem}

In view of Proposition \ref{Pr:Harnack}, we can argue as in \cite[Proof of Theorem 8.19]{GT}, obtaining
\begin{Prop}[Strong minimum principle]\label{Pr:smp}
Let $\omega \ge 0$ and assume that the Ricci curvature of $(M,g)$ is nonnegative. Assume, additionally, that $u \in H^1(M)$ satisfies $-\Delta u + \omega u \ge  0$. If there exists $\Omega \subset M$ open and connected and a geodesic ball $B \subset \Omega$ such that
\begin{equation*}
\inf_B u = \inf_\Omega u \le 0,
\end{equation*}
then $u$ is constant in $\Omega$.
\end{Prop}

\section{\texorpdfstring{$\varepsilon$}{epsilon}-approximations}\label{S:eps}

In this section, in addition to (\ref{AssV}), we consider the following stronger variants of (\ref{AssA}) and (\ref{AssB}):
\begin{itemize}
    \item[(\namedlabel{AssA1}{A1})] $\cA \in L^1(M) \cap L^{2N/(N+2)}(M)$, $\cA \geq 0$, and $\cA \not= 0$;
    \item[(\namedlabel{AssB1}{B1})] $\cB \in L^\infty (M)$ and $\cB_+ \not = 0$.
\end{itemize}

This section aims to prove the following weaker version of Theorem \ref{thm:main}. Then, in the next section, we will exploit the fact that $\cA$ and $\cB$ given in (\ref{AssA}) and (\ref{AssB}) can be approximated by functions satisfying (\ref{AssA1}) and (\ref{AssB1}).

\begin{Prop}\label{prop:main}
Assume that the Ricci curvature of $(M,g)$ is bounded below and \eqref{VolumesPositiveNew} holds. If (\ref{AssA1}), (\ref{AssB1}), (\ref{AssV}), and \eqref{eq:B+}--\eqref{eq:ThetaK} hold, then \eqref{eq:main} admits a nonnegative finite-energy supersolution $u \in H^1 (M)$. If, moreover, the Ricci curvature is nonnegative and $\cB \ge 0$, then $u$ is a positive solution to \eqref{eq:main}.
\end{Prop}

With the intention of proving Proposition \ref{prop:main}, consider, for $\varepsilon > 0$,  the auxiliary problem 
\begin{equation}\label{eq:appr}
-\Delta u + V(x)u = \cB(x) |u|^{2^*-2} u + \frac{\cA(x) u}{(\varepsilon + u^{2})^{2^*/2 + 1}}, \quad u \in H^1 (M).
\end{equation}
Analogously to Definition \ref{def:solution}, we say that $u\in H^1(M)$ solves \eqref{eq:appr} if, for every $\varphi\in \cC_0^\infty(M)$,
\[
\int_{M} \nabla u \nabla \varphi \,\dV
+ \int_{M} V(x)u \varphi \,\dV
=
\int_{M}\cB(x) |u|^{2^*-2} u \varphi \,\dV
+ \int_{M} \frac{\cA(x)u\varphi}{(\varepsilon+u^2)^{2^*/2+1}} \,\dV.
\]
Such solutions can be found as critical points of the functional $I_\varepsilon \colon H^1(M)\to\R$ given by
\begin{equation}
    \label{Ieps}
    I_\varepsilon(u)
:=
\frac{1}{2} \| u \|^2
-\frac{1}{2^*} \int_{M} \cB(x) |u|^{2^*} \,\dV
+\frac{1}{2^*} \int_{M} \frac{\cA(x)}{(\varepsilon + u^{2})^{2^*/2}} \,\dV.
\end{equation}

\begin{Lem}
For every $\varepsilon > 0$, $I_\varepsilon$ is of class $\cC^1$ and, for $u,v \in H^1(M)$,
\[
I_\varepsilon'(u)(v) = \int_{M} \left[\nabla u \nabla v + V(x) u v - \cB(x) |u|^{2^*-2} u v - \frac{\cA(x) u v}{(\varepsilon + u^2)^{2^*/2+1}} \right]\, \dV.
\]
\end{Lem}

\begin{proof}
We only need to prove the statement for the rightmost term in \eqref{Ieps}, as it is standard for the remaining ones.
Consider
$$
a(x,s) := \frac{\cA(x)}{(\varepsilon + s^2)^{2^*/2}}, \quad x \in M, \, s \in \R.
$$
Since $a$ is smooth with respect to $s$ and
\[
\partial_s a(x,s) = -2^* \frac{\cA(x) s}{(\varepsilon + s^2)^{2^*/2+1}},
\]
we first observe that, by (\ref{AssA1}) and the fact that $\R \ni s \mapsto s / (\varepsilon + s^2)^{2^*/2+1}\in\R$ is bounded, there exists $C>0$ such that, for every $u, v \in H^1(M)$, there holds
\begin{equation*}
\int_M \frac{\cA(x) |u v|}{(\varepsilon + u^2)^{2^*/2+1}} \, \dV
\le
C \int_{M} \cA(x) |v| \, \dV \le |\cA|_{L^{2N/(N+2)}} |v|_{L^{2^*}} < +\infty.
\end{equation*}

Now we prove the continuity of the Fr\'echet derivative. Let $u_n \to u$ in $H^1(M)$ and $v \in H^1(M)$ such that $\|v\| = 1$. Up to a subsequence, we can assume that $u_n \to u$ a.e. in $M$. Since $H^1(M) \hookrightarrow L^{2^*}(M)$, there exists $C>0$, independent of $v$, such that
\begin{equation*}
\int_M \cA(x) \left| \frac{u_n}{(\varepsilon + u_n^2)^{2^*/2+1}} - \frac{u}{(\varepsilon + u^2)^{2^*/2+1}} \right| |v| \, \dV \le C \left|\cA \cdot \left( \frac{u_n}{(\varepsilon + u_n^2)^{2^*/2+1}} - \frac{u}{(\varepsilon + u^2)^{2^*/2+1}} \right) \right|_{L^{2N/(N+2)}}.
\end{equation*}
Finally, again using the boundedness of $\R \ni s \mapsto s / (\varepsilon + s^2)^{2^*/2+1}\in\R$, the right-hand side above tends to $0$ in view of the dominated convergence theorem.
\end{proof}

Now, we will follow the argument of \cite{Hebey}. We write $I_\varepsilon = I^{(1)} + I_\varepsilon^{(2)}$, where
\begin{equation*}
I^{(1)}(u) = \frac{1}{2} \| u \|^2
-\frac{1}{2^*} \int_{M} \cB(x) |u|^{2^*} \,\dV \quad \text{and} \quad I_\varepsilon^{(2)} = \frac{1}{2^*} \int_{M} \frac{\cA(x)}{(\varepsilon + u^{2})^{2^*/2}} \,\dV.
\end{equation*}
We recall $\psi \in H^1(M)$, $\Theta > 0$, and $K > 0$ from Theorem \ref{thm:main}, and for the sake of simplicity, we denote
\begin{equation*}
b_\pm := |\cB_\pm|_{L^\infty(\supp \psi)}.
\end{equation*}
Additionally, for $t \ge 0$, we introduce the functions
\begin{equation*}
\Phi(t) = \frac12 t^2 - \frac{1}{2^*} S b_+ t^{2^*} \quad \text{and} \quad \Psi(t) = \frac12 t^2 + \frac{1}{2^*} S b_- t^{2^*}.
\end{equation*}
It is readily seen that $\Phi$ is increasing on $[0,t_0]$ and decreasing on $[t_0,+\infty)$, where
\begin{equation*}
t_0 = \frac{1}{(S b_+)^{(N-2)/4}}.
\end{equation*}
In particular,
\begin{equation*}
\Phi(t_0) = \frac{1}{N (S b_+)^{N/2-1}}.
\end{equation*}
Up to replacing $\psi$ with $\psi / \|\psi\|$, we can assume that $\|\psi\| = 1$.\\
Observe that
\begin{equation}\label{eq:>Phi}
I_\varepsilon(t_0 \psi) > I^{(1)}(t_0 \psi) \ge \frac{t_0^2}{2} - S b_+ \frac{t_0^{2^*}}{2^*} = \Phi(t_0).
\end{equation}
Moreover, if $t_1 := \Theta t_0 \in (0,t_0)$ (note that, from \eqref{eq:ThetaK}, $\Theta < \sqrt{2/N} < 1$), since
\[
\Psi(t_1)
= \frac{\Theta^2}{2} \frac{1}{(S b_+)^{N/2-1}} + \frac{\Theta^{2^*}}{2^*} S b_- \frac{1}{(S b_+)^{N/2}}
= \frac{N \Theta^2}{2} \left( 1 + \frac{N-2}{N} \Theta^{2^*-2} \frac{b_-}{b_+} \right) \Phi(t_0)
\]
and, from \eqref{eq:psiK},
\begin{equation*}
\int_{M} \frac{\cA(x)}{(t_1 |\psi|)^{2^*}} \, \dV \le 
\frac{N K}{\Theta^{2^*}} \Phi(t_0),
\end{equation*}
we have also that
\[
I_\varepsilon(t_1 \psi) 
< \Psi(t_1) + \frac{1}{2^* t_1^{2^*}} \int_{M} \frac{\cA(x)}{|\psi|^{2^*}} \, \dV 
\le \frac{1}{2}\left( N \Theta^2 + (N-2) \frac{b_-}{b_+} \Theta^{2^*} + (N-2) \frac{K}{\Theta^{2^*}} \right) \Phi(t_0).
\]
Hence, from \eqref{eq:ThetaK} and \eqref{eq:>Phi}, we obtain
\begin{equation*}
I_\varepsilon(t_0 \psi) 
> \Phi(t_0)
> I_\varepsilon(t_1 \psi).
\end{equation*}

Now we prove that $I_\varepsilon$ has the mountain pass geometry. First, since
\[
I_\varepsilon(t \psi)
\leq
\frac{t^2}{2}-\frac{t^{2^*}}{2^*} \int_{M} \cB(x) |\psi|^{2^*} \, \dV + \frac{1}{2^* t^{2^*}} \int_{M} \frac{\cA(x)}{|\psi|^{2^*}} \, \dV,
\]
and so $\lim_{t \to +\infty} I_\varepsilon(t \psi) = -\infty$, there exists $t_2 \in (t_0,+\infty)$ such that
\begin{equation}\label{eq:MPt}
I_\varepsilon(t_2 \psi) \le I_\varepsilon(t_1 \psi).
\end{equation}
Next, let us consider the set
\begin{equation*}
\Gamma = \Set{\gamma \in \cC([0,1],H^1(M)) | \gamma(0) = t_1 \psi \text{ and } \gamma(1) = t_2 \psi}.
\end{equation*}
If $u \in H^1(M)$ and $\|u\| = t_0$, then an argument similar to \eqref{eq:>Phi} shows that
\begin{equation*}
I_\varepsilon(u) > \Phi(t_0) > I_\varepsilon(t_1 \psi).
\end{equation*}
This proves the claim because every path in $\Gamma$ must cross the sphere $\Set{u \in H^1(M) | \|u\| = t_0}$, and we can define the mountain pass level
\begin{equation*}
m_\varepsilon = \inf_{\gamma \in \Gamma} \max_{s \in [0,1]} I_\varepsilon\bigl(\gamma(s)\bigr) > \Phi(t_0) > 0.
\end{equation*}

\begin{Rem}\label{rem:t}
Since, given $u \in H^1(M)$, $I_\varepsilon(u)$ is nonincresing with respect to $\varepsilon$, assuming that $\varepsilon$ varies in a fixed interval $(0,\varepsilon_0]$, we can refine \eqref{eq:MPt} to
\begin{equation}\label{eq:MPtREF}
I_\varepsilon(t_2 \psi) \le I(t_2 \psi) \le I_{\varepsilon_0}(t_1 \psi) \le I_\varepsilon(t_1 \psi).
\end{equation}
hence, we can take $t_2$ such that it does not depend on $\varepsilon$.\\
Next, assume that there exist nonnegative nondecreasing sequences $(\cA_n)_n$, $(\cB_{n,+})_n$, and $(\cB_{n,-})_n$ such that $\cA_n \to \cA$, $\cB_{n,+} \to \cB_+$, and $\cB_{n,-} \to \cB_-$ a.e. as $n \to +\infty$. Denoting $I^n$ and $I_{\varepsilon_0}^n$  the functionals $I$ and $I_{\varepsilon_0}$ where $\cB$ and $\cA$ are replaced with $\cB_{n,+} - \cB_{n,-}$ and $\cA_n$ respectively, we can then refine the middle inequality in \eqref{eq:MPtREF} to
\begin{align*}
I^n(t_2 \psi) & \le \frac{t_2^2}{2} - \frac{t_2^{2^*}}{2^*} \int_M \bigl(\cB_{1,+}(x) - \cB_-(x)\bigr) |\psi|^{2^*} \, \dV + \frac{1}{2^* t_2^{2^*}} \int_M \frac{\cA(x)}{|\psi|^{2^*}} \, \dV \\
& \le \frac{t_1^2}{2} - \frac{t_1^{2^*}}{2^*} \int_M \bigl(\cB_+(x) - \cB_{1,-}(x)\bigr) |\psi|^{2^*} \, \dV + \frac{1}{2^*} \int_M \frac{\cA_1(x)}{(\varepsilon_0 +|t_1 \psi|^2)^{2^*}} \, \dV \le I_{\varepsilon_0}^n(t_1 \psi)
\end{align*}
and obtain that $t_2$ does not depend on $n$ either.\\
These facts will be used in Section \ref{Sec:pfmain}.
\end{Rem}
Then, for the approximated problem, we have
\begin{Prop}\label{prop:EX}
For every $\varepsilon>0$ there exists a nonnegative solution $u_\varepsilon \in H^1(M)$ to \eqref{eq:appr}. Moreover, $u_\varepsilon \in \cC_\textup{loc}^{1,\alpha}(M)$ for all $\alpha \in (0,1)$.
\end{Prop}
\begin{proof}
Since $I_\varepsilon$ has the mountain pass geometry, there exists $(u_n)_n \subset H^1(M)$ such that
\begin{equation*}
\lim_n I_\varepsilon(u_n) = m_\varepsilon \quad \text{and} \quad \lim_n I_\varepsilon'(u_n) = 0.
\end{equation*}
Next, $(u_n)_n$ is bounded in $H^1(M)$ since, for $n$ sufficiently large,
\begin{equation}\label{nehari}
\begin{aligned}
m_\varepsilon + 1 + \|u_n\| &
\geq I_\varepsilon(u_n) - \frac{1}{2^*} I_\varepsilon'(u_n)(u_n)\\
& =
\frac{1}{N} \| u_n \|^2
+ \frac{1}{2^*} \int_{M} \frac{\cA(x)}{(\varepsilon + u_n^{2})^{2^*/2}} \,\dV
+ \frac{1}{2^*} \int_{M} \frac{\cA(x) u_n^2}{(\varepsilon + u_n^{2})^{2^*/2 + 1}} \,\dV
\geq \frac{1}{N} \| u_n \|^2.
\end{aligned}
\end{equation}
Thus, up to a subsequence, there exists $u_\varepsilon \in H^1(M)$ such that $u_n \rightharpoonup u_\varepsilon$ in $H^1(M)$ and $u_n \to u_\varepsilon$ a.e. in $M$. From the weak convergence and Vitali's theorem, we obtain that for every $v \in H^1(M)$
\[
\int_{M} \nabla u_n \nabla v + V(x) u_n v - \cB(x) |u_n|^{2^*-2} u_n v \, \dV \to \int_{M} \nabla u_\varepsilon \nabla v + V(x) u_\varepsilon v - \cB(x) |u_\varepsilon|^{2^*-2} u_\varepsilon v \, \dV,
\]
while from (\ref{AssA1}), the boundedness of $\R \ni s \mapsto s / (\varepsilon + s^2)^{2^*/2+1}\in\R$, and the dominated convergence theorem we get
\[
\int_{M} \frac{\cA(x) u_n v}{(\varepsilon + u_n^2)^{2^*/2+1}} \, \dV \to \int_{M} \frac{\cA(x) u_\varepsilon v}{(\varepsilon + u_\varepsilon^2)^{2^*/2+1}} \, \dV.
\]
Consequently, $u_\varepsilon$ is a solution to \eqref{eq:appr}. Replacing $u$ with $u_+:=\max\{u,0\}$ in the right hand-side of \eqref{eq:appr}, we can assume that $u_\varepsilon\geq 0$. A standard argument (cf., e.g., \cite[page 271]{Struwe}) yields that $u_\varepsilon \in \cC_\textup{loc}^{1,\alpha}(M)$ for all $\alpha \in (0,1)$.
\end{proof}

Having obtained a family of solutions to the problems \eqref{eq:appr}, now we study the properties of $(m_\varepsilon)_\varepsilon$ and $(u_\varepsilon)_\varepsilon$.

\begin{Lem}\label{lem:bdd}
$(m_\varepsilon)_\varepsilon$ and $(\|u_\varepsilon\|)_\varepsilon$ are bounded.
\end{Lem}
\begin{proof}
Passing to the limit in \eqref{nehari} and using the weak lower semicontinuity of the norm, we get
\begin{equation}\label{eq:bound1}
m_\varepsilon\geq \frac{1}{N} \| u_\varepsilon \|^2.
\end{equation}
Hence, it is enough to show that $\sup_{\varepsilon > 0} m_\varepsilon < +\infty$. Let us consider $\gamma \in \Gamma$ given by
\begin{equation*}
\gamma(s) = \bigl(st_2+(1-s)t_1\bigr) \psi, \quad s \in [0,1].
\end{equation*}
Then, recalling that $\Psi$ is increasing,
\begin{equation*}
I_\varepsilon\bigl(\gamma(\psi)\bigr) < \Psi\bigl(st_2+(1-s)t_1\bigr) + \frac{1}{2^* \bigl(st_2+(1-s)t_1\bigr)^{2^*}} \int_{M} \frac{\cA(x)}{|\psi|^{2^*}} \, \dV < \Psi(t_2) + \frac{1}{2^* t_1^{2^*}} \int_{M} \frac{\cA(x)}{|\psi|^{2^*}} \, \dV
\end{equation*}
for every $s \in [0,1]$, whence
\begin{equation}\label{eq:bound2}
m_\varepsilon \le \max_{s \in [0,1]} I_\varepsilon\bigl(\gamma(s)\bigr) < \Psi(t_2) + \frac{1}{2^* t_1^{2^*}} \int_{M} \frac{\cA(x)}{|\psi|^{2^*}} \, \dV
\end{equation}
for every $\varepsilon > 0$.
\end{proof}

\begin{Lem}\label{lem:pos}
If the Ricci curvature of $(M,g)$ is nonnegative, then there exists $\varepsilon_0 > 0$ such that for every $\varepsilon \in (0,\varepsilon_0]$ the function $u_\varepsilon$ obtained in Proposition \ref{prop:EX} is positive everywhere.
\end{Lem}
\begin{proof}
From Lemma \ref{lem:bdd} and the fact that $I_\varepsilon'(u_\varepsilon) (u_\varepsilon) = 0$, there exists $C>0$ such that for every $\varepsilon>0$
\begin{equation}\label{eq:epsC}
\int_M \frac{\cA(x)}{(\varepsilon + u_\varepsilon^2)^{2/2^*}} \, \dV \le C.
\end{equation}
If there existed a sequence $\varepsilon_j \to 0$ such that $u_{\varepsilon_j} = 0$ for all $j$, then \eqref{eq:epsC} would be false for every sufficiently large $j$. Next, we recall that $u_\varepsilon \in \cC_\textup{loc}^{1,\alpha}(M) \subset L_\textup{loc}^\infty(M)$ and rewrite \eqref{eq:appr} as
\[
-\Delta u_\varepsilon + V(x) u_\varepsilon + \cB_-(x) u_\varepsilon^{2^*-1} = \cB_+(x) u_\varepsilon^{2^*-1} + \frac{\cA(x) u_\varepsilon}{(\varepsilon + u_\varepsilon^{2})^{2^*/2 + 1}}.
\]
Take a geodesic ball $B \subset M$ and suppose by contradiction that $\essinf_B u_\varepsilon = 0$. Since $u_\varepsilon \geq 0$, $\essinf_M u_\varepsilon = 0$. Then, from Proposition \ref{Pr:smp}, $u_\varepsilon$ is constant on $M$, and therefore $u_\varepsilon = 0$, a contradiction. Hence, $\essinf_B u_\varepsilon > 0$.
\end{proof}

We will now prove the existence of a solution to the original problem by passing to the limit with $\varepsilon \to 0^+$. The challenging problem is to pass to the limit in the singular term, where we make use of the Harnack inequality.
\begin{proof}[Proof of Proposition \ref{prop:main}]

We divide the proof into three steps.

\textbf{Step 1: passing to the limit.} 
From Lemma \ref{lem:bdd}, we can assume that $u_\varepsilon \rightharpoonup u_0\in H^1(M)$ and $u_\varepsilon \to u_0$ a.e. in $ M$ as $\varepsilon \to 0^+$. Since $u_\varepsilon > 0$, we get $u_0 \geq 0$.
From \eqref{eq:epsC} and Fatou's lemma, we obtain that $u_0>0$ a.e. in the interior of $\supp \cA$. Indeed, if $\Omega
\subset\supp \cA$ with $\operatorname{vol}_g(\Omega)>0$,
$$
\int_{\Omega} \frac{\cA(x)}{u_0^{2^*}} \, \dV
\leq
\liminf_{\varepsilon \to 0^+} \int_{\Omega} \frac{\cA(x)}{(\varepsilon + u_\varepsilon^2)^{2^* / 2}} \, \dV
<
+\infty. 
$$
Due to to (\ref{AssB1}), $u_0$ has finite energy.

\textbf{Step 2: $u_0$ is a supersolution to \eqref{eq:main}.} Let $\varphi \in \cC_0^\infty (M)$. Recalling that
$$
\langle u_\varepsilon, \varphi \rangle = \int_{M} \nabla u_\varepsilon \nabla \varphi \,\dV
+ \int_{M} V(x)u_\varepsilon \varphi \,\dV,
$$
we know that
\begin{equation}\label{eq:u_e-sol}
\langle u_\varepsilon, \varphi \rangle = \int_{M}\cB(x) u_\varepsilon^{2^*-1} \varphi \,\dV
+ \int_{M} \frac{\cA(x) u_\varepsilon \varphi}{(\varepsilon+u_\varepsilon^2)^{2^*/2+1}} \,\dV.
\end{equation}
The weak convergence implies that $\langle u_\varepsilon, \varphi \rangle \to \langle u_0, \varphi \rangle$. From Vitali's convergence theorem, we know that
\begin{equation}\label{eq:B}
\int_{M}\cB(x) u_\varepsilon^{2^*-1} \varphi \,\dV \to \int_{M}\cB(x) u_0^{2^*-1} \varphi \,\dV.
\end{equation}
Assumming that $\varphi \geq 0$ and since $u_0>0$ a.e. in the interior of $\supp \cA$, applying Fatou's lemma again we get
\begin{align*}
+\infty>\liminf_{\varepsilon \to 0^+} \int_{M} \frac{\cA(x) u_\varepsilon \varphi}{(\varepsilon + u_\varepsilon^2)^{2^* / 2 + 1}} \, \dV
&= \liminf_{\varepsilon \to 0^+} \int_{\supp \varphi} \frac{\cA(x) u_\varepsilon \varphi}{(\varepsilon + u_\varepsilon^2)^{2^* / 2 + 1}} \, \dV \\
&\geq \int_{\supp \varphi} \frac{\cA(x) \varphi}{u_0^{2^*+1}} \, \dV
= \int_{M} \frac{\cA(x) \varphi}{u_0^{2^*+1}} \, \dV
\end{align*} 
and finally
\begin{equation*}
\langle u_0, \varphi \rangle \geq \int_{M}\cB(x) u_0^{2^*-1} \varphi \,\dV
+ \int_{M} \frac{\cA(x) \varphi}{u_0^{2^*+1}} \,\dV.
\end{equation*}

\textbf{Step 3: $u_0$ is a positive solution to \eqref{eq:main} if the Ricci curvature is nonnegative and $\cB \ge 0$.}
Arguing as in the proof of Lemma \ref{lem:pos}, we obtain that $u_0$ is positive. From Proposition \ref{prop:EX}, $u_\varepsilon \ge 0$ satisfies
\begin{equation*}
-\Delta u_\varepsilon + V(x) u_\varepsilon \geq 0.
\end{equation*}
Let $\varphi \in \cC_0^\infty(M)$. Since $\supp \varphi$ is compact, we can cover it with a finite number of geodesic balls. Let $B(r)$ be any of such balls and $\varepsilon_0$ the number given in Lemma \ref{lem:pos}. From Proposition \ref{Pr:Harnack}, there exist $C,q > 0$ such that, for every $\varepsilon \in (0,\varepsilon_0]$,
\begin{equation}\label{eq:Harnack}
\inf_{B(r)} u_\varepsilon \geq C \left( \int_{B(2r)} u_\varepsilon^q \, \dV \right)^{1/q}.
\end{equation}
There are two possibilities: either there exists $\overline \varepsilon \in (0,\varepsilon_0]$ such that
\begin{equation*}
\inf_{\varepsilon \in (0,\varepsilon_0]} \int_{B(2r)} u_\varepsilon^q \, \dV = \int_{B(2r)} u_{\overline \varepsilon}^q \, \dV > 0,
\end{equation*}
or, from Fatou's lemma,
\begin{equation*}
\inf_{\varepsilon \in (0,\varepsilon_0]} \int_{B(2r)} u_\varepsilon^q \, \dV = \liminf_{\varepsilon \to 0^+} \int_{B(2r)} u_\varepsilon^q \, \dV \ge \int_{B(2r)} u_0^q \, \dV > 0.
\end{equation*}
Either way,
\begin{equation*}
\inf_{\varepsilon \in (0,\varepsilon_0]} \inf_{B(r)} u_\varepsilon \ge C \inf_{\varepsilon \in (0,\varepsilon_0]} \left( \int_{B(2r)} u_\varepsilon^q \, \dV \right)^{1/q} > 0
\end{equation*}
Since such balls are finitely many,
$$\inf_{\varepsilon \in (0,\varepsilon_0]} \inf_{\supp \varphi} u_\varepsilon =: c > 0.$$
Thus, for every $\varepsilon > 0$ and a.e. $x \in M$,
\begin{equation*}
\frac{\cA(x) u_\varepsilon(x) \varphi(x)}{\bigl( \varepsilon + u_\varepsilon^2(x) \bigr)^{2^*/2+1}} \le \frac{\cA(x) \varphi(x)}{u_\varepsilon^{2^*+1}(x)} \le \frac{\cA(x) \varphi(x)}{c^{2^*+1}},
\end{equation*}
and the dominated convergence theorem yields
\begin{equation*}
\lim_{\varepsilon \to 0^+} \int_{M} \frac{\cA(x) u_\varepsilon \varphi}{(\varepsilon + u_\varepsilon^2)^{2^* / 2 + 1}} \, \dV = \int_{M} \frac{\cA(x) \varphi}{u_0^{2^*+1}} \, \dV.
\end{equation*}
This, \eqref{eq:u_e-sol}, \eqref{eq:B}, and the weak convergence $u_\varepsilon \rightharpoonup u_0$ as $\varepsilon \to 0^+$ conclude the proof.
\end{proof}

\begin{Rem}\label{rem:Linfinity}
If $\cB_- \not = 0$, then $u_\varepsilon$ satisfies
\begin{equation*}
-\Delta u_\varepsilon + \left( V(x) + \cB_-(x) u_\varepsilon^{2^*-2} \right) u_\varepsilon \ge 0.
\end{equation*}
Nevertheless, although $u_\varepsilon \in L^\infty(M)$ (from the classical regularity theory) and $(\|u_\varepsilon\|)_{\varepsilon \in (0,\varepsilon_0]}$ is bounded, a constant $c>0$ such that $|u_\varepsilon|_{L^\infty} \le c \|u_\varepsilon\|$ for every $\varepsilon \in (0,\varepsilon_0]$ may not exist, as $M$ is possibly noncompact and the right-hand side of \eqref{eq:appr} has a Sobolev-critical growth at infinity. For this reason, when $\cB_- \not = 0$, the existence of $C>0$ in \eqref{eq:Harnack} independent of $\varepsilon$ is not guaranteed, nor can we use Proposition \ref{Pr:smp}.
\end{Rem}

\subsection{An alternative to \texorpdfstring{\eqref{eq:ThetaK}}{(\ref{eq:ThetaK})}}\label{ss:alt}

When $\cB \not \ge 0$, $\Theta$ in \eqref{eq:ThetaK} could depend on $\psi$. Additionally, the larger $|\cB_-|_{L^\infty(\supp \psi)}$ is, the smaller $\Theta$ and, consequently, $K$ need to be for \eqref{eq:ThetaK} to hold, giving a narrower range for \eqref{eq:psiK}. For this reason, it is interesting to find an alternative to \eqref{eq:ThetaK}. To this aim, we define $\Phi$ and $\Psi$ in a slightly different way, i.e.,
\begin{equation*}
\Phi(t) = \frac12 t^2 - \frac{1}{2^*} S |\cB|_{L^\infty(\supp \psi)} t^{2^*} \quad \text{and} \quad \Psi(t) = \frac12 t^2 + \frac{1}{2^*} S |\cB|_{L^\infty(\supp \psi)} t^{2^*}.
\end{equation*}
We need a slightly stronger version of \eqref{eq:psiK}, that is,
\begin{equation*}
\|\psi\|^{2^*} \int_{M} \frac{\cA(x)}{|\psi|^{2^*}} \, \dV \leq \frac{K}{(|\cB|_{L^\infty(\supp \psi)} S)^{N-1}}.
\end{equation*}
Then, proceeding in a way similar to what was done earlier, we obtain
\begin{equation*}
I_\varepsilon(t_1 \psi) < \left( \frac{N}{2} \Theta^2 + \frac{N-2}{2} \Theta^{2^*} + \frac{N-2}{2} \frac{K}{\Theta^{2^*}} \right) \Phi(t_0),
\end{equation*}
where $t_0 = (S|\cB|_{L^\infty(\supp \psi)})^{(2-N)/4}$. Considering \eqref{eq:>Phi}, a natural condition to impose is
\begin{equation*}
N \Theta^2 + (N-2) \Theta^{2^*} + (N-2) \frac{K}{\Theta^{2^*}} \le 2.
\end{equation*}

\section{Proof of the main results}\label{Sec:pfmain}

We begin by proving the existence of a finite-energy solution to \eqref{eq:main} under the weaker assumptions (\ref{AssA}) and (\ref{AssB}).
\begin{proof}[Proof of Theorem \ref{thm:main}]
Let $(\cA_n)_n \subset L^1(M) \cap L^{2N/(2+N)}(M)$ and $(\cB_n)_n\subset L^\infty(M)$ such that $\cA_n \ge 0$, $\cA_n \le \cA_{n+1}$, $\cB_{n,+} \le \cB_{n+1,+}$, $\cB_{n,-} \le \cB_{n+1,-}$, where $\cB_{n,+}:=\max\{0,\cB_n\}$ and $ \cB_{n,-}:=\max\{0,-\cB_n\}$, and
\begin{equation*}
\cA_n \to \cA \text{ in } L^1_\textup{loc}(M) \quad \text{and} \quad \cB_n \to \cB \text{ in } L^{2^*}_\textup{loc}(M) 
\end{equation*}
as $n \to +\infty$.\\ 
Observe that, for $\cA_n$ and $\cB_n$, \eqref{eq:psiK} holds for all $n$ and \eqref{eq:B+} holds for every sufficiently large $n$. If $\cB_- = 0$ in $\supp \psi$, then \eqref{eq:ThetaK} does not depend on $\cB_n$ or $\cB$. If, instead, $|\cB_-|_{L^\infty(\supp \psi)} > 0$, then, up to replacing $\cB_{n,-}$ with $(1-\frac1n) \cB_{n,-}$, we can assume that $|\cB_{n,-}|_{L^\infty(\supp \psi)} < |\cB_-|_{L^\infty(\supp \psi)}$. Thus, for every $n$, there exists $k_n$ such that
\begin{equation*}
\frac{|\cB_{n,-}|_{L^\infty(\supp \psi)}}{|\cB_{k_n,+}|_{L^\infty(\supp \psi)}} \le \frac{|\cB_-|_{L^\infty(\supp \psi)}}{|\cB_+|_{L^\infty(\supp \psi)}},
\end{equation*}
and \eqref{eq:ThetaK} holds up to suitably replacing the sequence $(\cB_{n,+},\cB_{n,-})_n$.\\
For every $n$, let $u_n$ be the supersolution to \eqref{eq:main} given by Proposition \ref{prop:main} with $\cA_n$ and $\cB_n$ instead of $\cA$ and $\cB$. We shall prove that $(u_n)_n$ is bounded. Denote
\begin{align}
\Psi_n(t) & = \frac12 t^2 + \frac{1}{2^*} S |\cB_{n,-}|_{L^\infty(\supp \psi)} t^{2^*} \le \frac12 t^2 + \frac{1}{2^*} S |\cB_-|_{L^\infty(\supp \psi)} t^{2^*}, 
\label{Psin}
\\
t_{0,n} & = \frac{1}{(S |\cB_{n,+}|_{L^\infty(\supp \psi)})^{(N-2)/4}} \ge \frac{1}{(S |\cB_+|_{L^\infty(\supp \psi)})^{(N-2)/4}},
\label{t0n}
\end{align}
and
$t_{1,n} = \Theta t_{0,n}$.
Let us recall from Section \ref{S:eps} that each $u_n$ is obtained as the weak limit of a family of functions $(u_{n,\varepsilon})_{\varepsilon}$ as $\varepsilon \to 0^+$ and that each $u_{n,\varepsilon}$ is, in turn, obtained as the weak limit of a Palais--Smale sequence for corresponding functional $I_{n,\varepsilon}$ at the level $m_{n,\varepsilon}$.

From \eqref{eq:bound1}, \eqref{eq:bound2}, Remark \ref{rem:t}, \eqref{Psin}, and \eqref{t0n},
\begin{equation}\label{eq:bdd}
\begin{split}
\frac1N \|u_n\|^2
& \le \frac1N \liminf_{\varepsilon \to 0^+} \|u_{n,\varepsilon}\|^2
\le \liminf_{\varepsilon \to 0^+} m_{n,\varepsilon} \le \Psi_n(t_2) + \frac{1}{2^* t_{1,n}^{2^*}} \int_{M} \frac{\cA(x)}{|\psi|^{2^*}} \, \dV \\
& \le \frac12 t_2^2 + \frac{1}{2^*} S |\cB_-|_{L^\infty(\supp \psi)} t_2^{2^*} + \frac{(S|\cB_+|_{L^\infty(\supp \psi)})^{N/2}}{2^* \Theta^{2^*}} \int_{M} \frac{\cA(x)}{|\psi|^{2^*}} \, \dV. 
\end{split}
\end{equation}
Since $(u_n)_n$ is bounded, there exists $u_0 \in H^1(M)$ such that $u_n \rightharpoonup u_0$ in $H^1(M)$ and $u_n \to u_0$ a.e. in $M$ as $n \to +\infty$ along a subsequence. In particular, $u_0 \ge 0$. Now, let $\varphi \in C_0^\infty(M)$ nonnegative. From the weak convergence, Fatou's lemma, and the fact that $\lim_n |\cB_n \varphi - \cB \varphi|_{2^*} = 0$, we obtain
\begin{align*}
\int_M \nabla u_0 \nabla\varphi + V(x) u_0 \varphi - \cB(x) u_0^{2^*-1} \varphi \, \mathrm{d} v_g
& = \lim_n \int_M \nabla u_n \nabla\varphi + V(x) u_n \varphi - \cB_n(x) u_n^{2^*-1} \varphi \, \mathrm{d} v_g \\
& \ge \lim_n \int_M \frac{\cA_n(x) \varphi}{u_n^{2^*+1}} \, \mathrm{d} v_g = \lim_n \int_{\supp \cA} \frac{\cA_n(x) \varphi}{u_n^{2^*+1}} \, \dV \\
& \ge \int_{\supp \cA} \frac{\cA(x) \varphi}{u_0^{2^*+1}} \, 
= \int_M \frac{\cA(x) \varphi}{u_0^{2^*+1}} \, \mathrm{d} v_g.
\end{align*}
To demonstrate that $u_0$ has finite energy, observe that from Proposition \ref{prop:EX} and \eqref{eq:bdd} there exists $C>0$ such that, for every $n$ and every $\varepsilon \in (0,\varepsilon_0]$,
\begin{align*}
C & \ge I_{n,\varepsilon}(u_{n,\varepsilon}) - \frac12 I_{n,\varepsilon}'(u_{n,\varepsilon})(u_{n,\varepsilon}) \\
& = \frac1N \int_{M} \cB_n(x) u_{n,\varepsilon}^{2^*} \, \dV + \frac{1}{2^*} \int_{M}\frac{\cA_n(x)}{(\varepsilon + u_{n,\varepsilon}^2)^{2^*/2}} \, \dV + \frac12 \int_{M} \frac{\cA_n(x) u_{n,\varepsilon}^2}{(\varepsilon + u_{n,\varepsilon}^2)^{2^*/2+1}} \, \dV \\
& \ge \frac1N \int_{M} \cB_n(x) u_{n,\varepsilon}^{2^*} \, \dV + \frac{1}{2^*} \int_{M}\frac{\cA_n(x)}{(\varepsilon + u_{n,\varepsilon}^2)^{2^*/2}} \, \dV,
\end{align*}
where we renamed $I_\varepsilon$ as $I_{n,\varepsilon}$. Finally, using Fatou's lemma, we first let $\varepsilon \to 0^+$ and then $n \to +\infty$.\\
We can now conclude as in the proof of Proposition \ref{prop:main}.
\end{proof}

Finally, we prove the nonexistence of nonnegative supersolutions, showing that the integrability of $\cA |\psi|^{-2^*}$ is, to some extent, an optimal assumption.

\begin{proof}[Proof of Theorem \ref{th:non}]
Assume by contradiction that $u \in H^1(M)$ as in the statement exists. Observe that, if $\cB \ge 0$ in $\supp u$, or if $\cB_- \in L^\infty(\supp u)$, the Ricci curvature is bounded below, and \eqref{VolumesPositiveNew} holds (using the embedding $H^1(M) \hookrightarrow L^{2^*}(M)$), then we get
\[
0\leq \int_{M} \cB_-(x) u^{2^*} \, \dV <+\infty.
\]
Let $(\varphi_n)_n \subset \cC_0^\infty(M)$ such that $\varphi_n \ge 0$ for every $n$ and $\varphi_n \to u$ in $H^1(M)$ as $n \to +\infty$. Since $u$ is a supersolution to \eqref{eq:main}, Fatou's lemma yields
\begin{align*}
+\infty & > \|u\|^2 + \int_{M} \cB_-(x) u^{2^*} \, \dV
= \lim_n \int_{M} \left[ \nabla u \nabla \varphi_n + V(x) u \varphi_n + \cB_-(x) u^{2^*-1} \varphi_n \right] \, \dV \\
& \ge \liminf_n \int_{M} \frac{\cA(x) \varphi_n}{u^{2^*+1}} \, \dV  \ge \int_{M} \frac{\cA(x)}{u^{2^*}} \, \dV,
\end{align*}
which is a contradiction.
\end{proof}

\section*{Acknowledgements}
Bartosz Bieganowski is supported by the National Science Centre, Poland (Grant no. 2022/47/D/\ ST1/00487). 
Pietro d'Avenia and Jacopo Schino are members of GNAMPA (INdAM) and were supported by GNAMPA-INdAM Project 2026 (CUP E53C25002010001).
Pietro d'Avenia is also supported by the Italian Ministry of University and Research under the Program Department of Excellence L. 232/2016 (CUP D93C23000100001). This paper has been completed during visits to the Politecnico di Bari and the University of Warsaw. The authors thank both institutions for the warm hospitality.

\section*{Competing interests} The authors have no competing interests to declare that are relevant to the content of this article.

\section*{Data availability} Not applicable.

\bibliography{bibliography}
\bibliographystyle{abbrv}

\end{document}